\documentclass[12pt,english]{amsart}

\input xy
\xyoption{all}

\usepackage{comment}
\usepackage[latin1]{inputenc}
\usepackage{amsfonts,amsmath,amssymb,amsthm}
\usepackage{url}   

\newcommand{\Zz}{\mathbb{Z}}
\newcommand{\Qq}{\mathbb{Q}}

{\theoremstyle{plain}
\newtheorem{theorem}{Theorem}[section]    

\newtheorem{twisting lemma}[theorem]{Twisting lemma}

\newtheorem{lemma}[theorem]{Lemma}       
\newtheorem{proposition}[theorem]{Proposition}  

}
{\theoremstyle{remark}
\newtheorem{remark}[theorem]{Remark}   

}

\def\Gal{\hbox{\rm Gal}}

\def\cm{\hbox{\hbox{\rm C}\kern-5pt{\raise 1pt\hbox{$|$}}}}

\def\lhfl#1#2{\smash{\mathop{\hbox to 12mm{\leftarrowfill}}
\limits^{#1}_{#2}}}

\def\rhfl#1#2{\smash{\mathop{\hbox to 12mm{\rightarrowfill}}
\limits^{#1}_{#2}}}

\def\build#1_#2^#3{\mathrel{
\mathop{\kern 0pt#1}\limits_{#2}^{#3}}}

\def\htrait#1#2{\smash{\mathop{\hbox to 12mm{\hrulefill}}
\limits^{#1}_{#2}}}

\def\sxbullet{{\raise 2pt\hbox{\bf .}}}
\title{A note on prime divisors of polynomials $P(T^k), k \geq 1$}

\author{Fran\c cois Legrand}

\email{legrandfranc@tx.technion.ac.il}

\address{Department of Mathematics, Technion - Israel Institute of Technology, Haifa 32000, Israel}

\date{\today}

\begin{document}

\maketitle

\begin{abstract}
Let $F$ be a number field, $O_F$ the integral closure of $\Zz$ in $F$ and $P(T) \in O_F[T]$ a monic separable polynomial such that $P(0) \not=0$ and $P(1) \not=0$. We give precise sufficient conditions on a given positive integer $k$ for the following condition to hold: there exist infinitely many non-zero prime ideals $\mathcal{P}$ of $O_F$ such that the reduction modulo $\mathcal{P}$ of $P(T)$ has a root in the residue field $O_F/\mathcal{P}$, but the reduction modulo $\mathcal{P}$ of $P(T^k)$ has no root in $O_F/\mathcal{P}$. This makes a result from a previous paper (motivated by a problem in field arithmetic) asserting that there exist (infinitely many) such integers $k$ more precise.
\end{abstract}

\section{Introduction}

Given a number field $F$, recall that a (non-zero) prime ideal $\mathcal{P}$ of the integral closure $O_F$ of $\Zz$ in $F$ is a {\it{prime divisor}} of a monic separable polynomial $P(T) \in O_F[T]$ if the reduction of $P(T)$ modulo $\mathcal{P}$ has a root in the residue field $O_F/\mathcal{P}$. Examples of standard results on this classical number theoretic notion, which may be obtained as applications of the Chebotarev density theorem, are (1) any monic separable polynomial $P(T) \in O_F[T]$ has infinitely many prime divisors and (2) if $P(T)$ is irreducible over $F$, then all but finitely many prime ideals of $O_F$ are prime divisors of $P(T)$ if and only if $P(T)$ has degree 1. See {\it{e.g.}} \cite{Nag69}, \cite{Sch69} and \cite{GB71} for more classical results.

Prime divisors of polynomials are known to appear in several areas of interest. For instance, in field arithmetic, the set of prime divisors of the product of the irreducible polynomials over $F$ of the branch points of a finite Galois extension $E/F(T)$ with $F$ algebraically closed in $E$ corresponds to the set of all prime ideals of $F$ that ramify in at least one specialization of $E/F(T)$; see \cite{Bec91} and \cite{Leg16c}. Another application deals with the following classical question in diophantine geometry: does a given superelliptic curve $C:y^n=P(t)$ over $F$ has twists without non-trivial rational points? Under the {\it{abc}}-conjecture, Granville \cite{Gra07} proved that, in the case of hyperelliptic curves over $\Qq$, only a few (quadratic) twists of $C$ have non-trivial rational points, provided that $C$ has genus at least 3. Partial unconditional results were then given by Sadek \cite{Sad14} (in the case of hyperelliptic curves over $\Qq$) and the author \cite{Leg16b}, asserting that, to guarantee the existence of twists of $C$ without non-trivial rational points, it suffices that $n$ divides the degree of $P(T)$ and there exist infinitely many prime ideals of $O_F$ which are not prime divisors of $P(T)$.

A more recent application in field arithmetic consists in producing examples of ``non-parametric extensions" over number fields \cite{Leg16c}, {\it{i.e.}}, finite Galois extensions $E/F(T)$ with Galois group $G$ and $F$ algebraically closed in $E$ which do not provide all the Galois extensions of $F$ with Galois group $G$ by specializing the indeterminate $T$. For example, the main result of \cite{Leg16a} asserts that such an extension of $F(T)$ with given non-trivial Galois group $G$ exists, unless $G$ is the Galois group of no Galois extension $E/F(T)$ with $F$ algebraically closed in $E$. See also \cite{Deb16} for other results about non-parametric extensions.

The proof of the main result of \cite{Leg16a} requires the following statement on prime divisors of polynomials, whose an ineffective proof may be found in \cite[\S4]{Leg16a}.

\vspace{2.4mm}

\noindent
{\bf{Proposition 1.}} \cite[Proposition 3.5]{Leg16a}
{\it{Assume $P(0) \not=0$ and $P(1) \not=0$. Then there exist infinitely many positive integers $k$ such that the following condition holds:

\vspace{0.5mm}

\noindent
{\rm{($*$/$k$)}} there exist infinitely many prime ideals of $O_F$ each of which is a prime divisor of $P(T)$ but not of $P(T^{k})$.}}

\vspace{2.4mm}

\noindent
Note that the result fails trivially if either $P(0)=0$ or $P(1)=0$. 

The aim of the present note is to make Proposition 1 effective by elementary techniques, thus allowing to make previous applications partially effective as well (in particular, that to superelliptic curves). We give precise sufficient conditions on a given positive integer $k$, which depend on the nature of the roots of the polynomial $P(T)$, for condition {\rm{($*$/$k$)}} to hold. In particular, we provide upper bounds on the number of distinct prime factors of such an integer $k$ and the multiplicity of each of them. This is motivated by the trivial fact that condition {\rm{($*$/$k$)}} holds for every multiple $k \geq 1$ of any given integer $k_0 \geq 1$, as soon as condition {\rm{($*$/$k_0$)}} holds. Proposition 2 below, which is Theorem \ref{thm 2} in the sequel, gives an idea of the effectiveness of our method. 

\vspace{2.4mm}

\noindent
{\bf{Proposition 2.}} {\it{Assume $P(0) \not=0$ and no root of unity is a root of $P(T)$. Then there exists an effective positive constant $c$ such that condition {\rm{($*$/$k$)}} of Proposition 1 holds for any positive integer $k$ which has at least one prime factor $p \geq c$.}}

\vspace{2.4mm}

\noindent
See \S2.4 for the definition of $c$, as well as Theorems \ref{thm 1} and \ref{thm 3} for the case where each root of $P(T)$ is a root of unity and the ``mixed" case. See also \S4 where we show that the conclusion of Proposition 2 fails in general for polynomials with roots of unity among their roots.

\vspace{2.4mm}

{\bf{Acknowledgments.}} This work is partially supported by the Israel Science Foundation (grants No. 696/13, No. 40/14 and No. 577/15).

\section{Statements of Theorems \ref{thm 1}, \ref{thm 2} and \ref{thm 3}}

This section is organized as follows. In \S2.1-4, we state the needed notation for Theorems \ref{thm 1}, \ref{thm 2} and \ref{thm 3}, which are then given in \S2.5.

\subsection{General notation}

For any number field $L$, $O_L$ is the integral closure of $\Zz$ in $L$ and, given a non-zero prime ideal $\mathcal{P}$ of $O_L$, $v_\mathcal{P}$ is the associated valuation over $L$. 

Given a number field $F$, let $P(T) \in O_F[T]$ be a monic separable polynomial such that $P(0) \not=0$ and $P(1) \not=0$. Let

\noindent
- $L_1$ be the splitting field of $P(T)$ over $F$, 

\noindent
- $r_1$ the number (possibly zero) of roots of $P(T)$ that are roots of unity,  

\noindent
- $r_2$ the number (possibly zero) of roots of $P(T)$ that are units of $O_{L_1}$, but not roots of unity, 

\noindent
-  {\hbox{$r_3$ the number (possibly zero) of roots of $P(T)$ that are not units of $O_{L_1}$.}}

\noindent
Set $r=r_1+r_2+r_3>0.$ We denote the (distinct) roots of $P(T)$ by 
$$t_1,\dots,t_{r_1}, t_{r_1+1}, \dots, t_{r_1 + r_2}, t_{r_ 1 + r_2+1}, \dots, t_{r_ 1 + r_2 + r_3 }=t_r$$ and assume that 

\noindent
- $t_1,\dots,t_{r_1}$ are roots of unity (if $r_1 >0$),

\noindent
- $t_{r_1+1}, \dots,t_{r_1 + r_2}$ are units of $O_{L_1}$, but not roots of unity (if $r_2 >0$),

\noindent
- $t_{r_1 + r_2+1}, \dots,t_r$ are not units of $O_{L_1}$ (if $r_3 >0$).

Pick a positive integer $k_0$ such that, for each prime number $p \geq k_0$, the fields $\mathbb{Q}(e^{2i \pi /p})$ and $L_1$ are linearly disjoint over $\Qq$ \footnote{Since the extension $\mathbb{Q}(e^{2i \pi /p})/\Qq$ totally ramifies at $p$, the fields $\mathbb{Q}(e^{2i \pi /p})$ and $L_1$ are linearly disjoint over $\Qq$ as soon as $p$ does not ramify in the extension $L_1/\Qq$.}. Finally, set
$$ d_1= \frac{|\bigcup_{j=1}^{r} \Gal(L_1/F(t_j))|}{|\Gal(L_1/F)|}>0.$$

\subsection{Data associated with $t_1, \dots, t_{r_1}$}
Assume $r_1 > 0$. For each $j \in \{1,\dots,r_1\}$, $t_j$ is a root of unity and $t_j \not= 1$ (as $P(1) \not=0$). Then there exist two coprime integers $m_j$ and $n_j$ such that $1 \leq m_j < n_j$ and $t_j=e^{2i \pi m_j/n_j}.$
Denote the set of all prime factors of $n_j$ by $\mathcal{S}_j$,
the smallest element of $\mathcal{S}_1 \cup \cdots \cup \mathcal{S}_{r_1}$ by 
$p_{\rm{min}}$ and, by $A_0,$ we mean
the smallest positive integer $A$ that satisfies

$$A > \frac{\log ([F:\Qq]) + \log(r_1) - \log(d_1)}{\log(p_{\rm{min}})}.$$

\subsection{Data associated with $t_{r_1+1}, \dots, t_{r_1 + r_2}$}
Assume $r_2 > 0$. Let $\{u_1, \dots, u_{v}\}$ be a system of fundamental units of $O_{L_1}$, {\it{i.e.}}, $u_1, \dots, u_{v}$ are units of $O_{L_1}$ such that each unit $u$ of $O_{L_1}$ can be uniquely written as
$u=\zeta \cdot u_1^{a_1} \cdots u_v^{a_{v}},$
where $\zeta$ is a root of unity and $a_1, \dots, a_{v}$ are integers.

For each $j \in \{r_1+1, \dots,r_1 + r_2\}$, set $t_j=\zeta_j \cdot u_1^{a_{j,1}} \cdots u_v^{a_{j,v}},$
where $\zeta_j$ is a root of unity and ${a_{j,1}}, \dots, {a_{j,v}}$ are integers. As $t_j$ is not a root of unity, one has $|{a_{j,l_j}}| \geq 1$ for some $l_j \in \{1,\dots,v\}$. Finally, set 
$$a_j = {\rm{gcd}}({|a_{j,1}}|, \dots, |{a_{j,v}}|) \in \mathbb{N} \setminus \{0\}$$ and
$a_0={\rm{lcm}}  (a_{r_1+1}, \dots,a_{r_1 + r_2}).$

\subsection{Data associated with $t_{r_1 + r_2+1}, \dots, t_{r}$} Assume $r_3 >0$. Given $j \in \{r_1 + r_2+1,\dots,r\}$, one has $t_j \not=0$ (as 0 is not a root of $P(T)$) and $t_j$ is an element of $O_{L_1}$ which is not a unit. Pick a non-zero prime ideal $\mathcal{P}_j$ of $O_{L_1}$ such that $v_{\mathcal{P}_j}(t_j)$
is a positive integer. Finally, set 
$v_0 = {\rm{lcm}} (v_{\mathcal{P}_{r_1 + r_2+1}}(t_{r_1 + r_2+1}), \dots, v_{\mathcal{P}_{r}}(t_{r}))$
and, in the case where $r_2 + r_3 >0$, we define the integer $c$ as follows:

\noindent
- $c= {\rm{max}} \, (r_2+r_3+1, k_0, {\rm{lcm}} (a_0, v_0) + 1)$ if $r_2 >0$ and $r_3 >0$,

\noindent
- $c= {\rm{max}} \, (r_3+1,  v_0 + 1)$ if $r_2 =0$,

\noindent
- $c={\rm{max}} \, (r_2+1, k_0, a_0 + 1)$ if $r_3=0$.

\subsection{Statements of Theorems \ref{thm 1}, \ref{thm 2} and \ref{thm 3}}

First, we consider the case where each root of $P(T)$ is a root of unity.

\begin{theorem} \label{thm 1}
Assume $P(1) \not=0$ and each root of $P(T)$ is a root of unity. Then condition {\rm{($*/{\rm{lcm}} (p_1^{A_0}, \dots, p_{r_1}^{A_0})$)}} of Proposition 1 holds for every $r_1$-tuple $(p_1,\dots,p_{r_1})$ of prime numbers in $\mathcal S_1 \times \cdots \times \mathcal S_{r_1}$.
\end{theorem}

Next, we handle the case where no root of $P(T)$ is a root of unity.

\begin{theorem} \label{thm 2}
Assume $P(0) \not=0$ and no root of unity is a root of $P(T)$. Then condition {\rm{($*/p$)}} of Proposition 1 holds for every prime number $p \geq c$.
\end{theorem}

Finally, we deal with the mixed case.

\begin{theorem} \label{thm 3}
Assume $P(0) \not=0$, $P(1) \not=0$, $P(T)$ has a root that is a root of unity and $P(T)$ has a root that is not a root of unity. Then, given a prime number $p_0 \geq c$, there exists an integer $A_0(p_0) \geq 1$ \footnote{See \S3.3.3 for the precise definition of $A_0(p_0)$.} such that condition {\rm{($*/{\rm{lcm}} (p_0, p_1^{A_0(p_0)}, \dots, p_{r_1}^{A_0(p_0)})$)}} of Proposition 1 holds for every $(p_1,\dots,p_{r_1}) \in \mathcal S_1 \times \cdots \times \mathcal S_{r_1}$.
\end{theorem}

\section{Proofs of Theorems \ref{thm 1}, \ref{thm 2} and \ref{thm 3}}

\subsection{Notation}

Given a positive integer $k$, denote
the splitting field of $P(T^k)$ over $F$ by $L_k$. For each $j \in \{1,\dots,r\}$, fix a $k$-th root
$\sqrt[k]{t_j}$ of $t_j$. If $r_1 >0$ and $j \in \{1,\dots,r_1\}$, we choose $\sqrt[k]{t_j} = e^{2i \pi m_j/(k \cdot n_j)}$. Finally, set

$$f(k)= \frac{|\bigcup_{j=1}^{r} \bigcup_{l=0}^{k-1} {\rm{Gal}}(L_k/F(e^{2i\pi l/k} \sqrt[k]{t_j}))|}{|\bigcup_{j=1}^{r} {\rm{Gal}}(L_k/F(t_j))|} \leq 1,$$
$$f_1(k)= \left \{ \begin{array} {ccc}
          \displaystyle{\frac{|\bigcup_{j=1}^{r_1} \bigcup_{l=0}^{k-1} {\rm{Gal}}(L_k/F(e^{2i\pi l/k} \sqrt[k]{t_j}))|}{|\bigcup_{j=1}^{r} {\rm{Gal}}(L_k/F(t_j))|}} & {\rm{if}} & r_1 >0 \\
          0 & {\rm{if}} &  r_1=0 ,\\   
          \end{array} \right.$$
$$f_2(k)= \left \{ \begin{array} {ccc}
          \displaystyle{\frac{|\bigcup_{j=r_1 +1}^{r} \bigcup_{l=0}^{k-1}{\rm{Gal}}(L_k/F(e^{2i\pi l/k} \sqrt[k]{t_j}))|}{|\bigcup_{j=1}^{r} {\rm{Gal}}(L_k/F(t_j))|}} & {\rm{if}} & r_2 + r_3 >0 \\
          0 & {\rm{if}} &  r_2 +r _3=0. \\   
          \end{array} \right.$$

\subsection{Statement of an auxiliary result}

Theorems \ref{thm 1}, \ref{thm 2} and \ref{thm 3} rest essentially on Theorem \ref{thm 4} below.

\begin{theorem} \label{thm 4}
Let $k$ be a positive integer.

\vspace{1mm}

\noindent
{\rm{(1)}} \cite[Lemma 2.2]{Leg16a} Condition {\rm{($*$/$k$)}} holds if and only if $f(k) <1$.

\vspace{1mm}

\noindent
{\rm{(2)}} Assume $r_1 >0$. Let $(p_1, \dots, p_{r_1}) \in \mathcal{S}_1 \times \cdots \times \mathcal{S}_{r_1}$, $\epsilon$ be a positive real number and $A$ a positive integer such that 
$$A > \frac{\log ([F:\Qq]) + \log(r_1) - \log(d_1) - \log(\epsilon)}{\log(p_{\rm{min}})}.$$
Then one has $f_1(k) < \epsilon$ if $k$ is a multiple of ${\rm{lcm}} (p_1^A, \dots, p_{r_1}^A)$. 

\vspace{1mm}

\noindent
{\rm{(3)}} {\hbox{Assume $r_2 + r_3 >0$. Then one has $f_2(k) <1$ if $k$ is a prime $\geq c$.}}
\end{theorem}

Theorem \ref{thm 4} is proved in \S3.4-5.

\subsection{Proofs of Theorems \ref{thm 1}, \ref{thm 2} and \ref{thm 3} under Theorem \ref{thm 4}}

\subsubsection{Proof of Theorem \ref{thm 1}}

Given $(p_1, \dots, p_{r_1}) \in \mathcal{S}_1 \times \cdots \times \mathcal{S}_{r_1}$, set $k={\rm{lcm}} (p_1^{A_0}, \dots, p_{r_1}^{A_0})$. By the definition of $A_0$, we may apply part (2) of Theorem \ref{thm 4} with $\epsilon=1$ and $A=A_0$ to get $f_1(k) <1$. As each root of $P(T)$ is a root of unity, one has $f_1(k)=f(k)$. Hence $f(k) < 1$. It then remains to apply part (1) of Theorem \ref{thm 4} to get Theorem \ref{thm 1}.

\subsubsection{Proof of Theorem \ref{thm 2}}

Assume that $p$ is a prime number $\geq c$. Then we may apply part (3) of Theorem \ref{thm 4} to get $f_2(p) <1$. As no root of unity is a root of $P(T)$, one has $f_2(p)=f(p)$. This gives $f(p) <1$. It then remains to apply part (1) of Theorem \ref{thm 4} to conclude.

\subsubsection{Proof of Theorem \ref{thm 3}}

Let $p_0$ be a prime number satisfying $p_0 \geq c$. We may then apply part (3) of Theorem \ref{thm 4} to get $f_2(p_0) <1$. Let
$A_0(p_0)$ be the smallest positive integer $A$ that satisfies
$$A > \frac{\log ([F:\Qq]) + \log(r_1) - \log(d_1) - \log(1-f_2(p_0))}{\log(p_{\rm{min}})}.$$ 

Given $(p_1,\dots,p_{r_1}) \in \mathcal{S}_1 \times \cdots \times \mathcal{S}_{r_1}$, set $k = {\rm{lcm}} (p_0, p_1^{A_0(p_0)}, \dots, p_{r_1}^{A_0(p_0)})$. By the definition of $A_0(p_0)$, we may apply part (2) of Theorem \ref{thm 4} with $\epsilon =1- f_2(p_0)$ and $A=A_0(p_0)$ to get $f_1(k) < 1- f_2(p_0)$. As $p_0$ divides $k$, one has $f_2(k) \leq f_2(p_0)$.
Hence
$$f(k) \leq f_1(k) + f_2(k) < 1- f_2(p_0) + f_2(p_0) = 1.$$
It then remains to apply part (1) of Theorem \ref{thm 4} to conclude.

\subsection{Proof of part (2) of Theorem \ref{thm 4}}

Denote the Euler function by $\varphi$. From now on, we assume that $r_1$ is positive.

Given an integer $n \geq 2$, set
$$h_k(n)  = \prod_{\substack{p | k \\ p | n}} p^{v_p(k)}.$$

The easy lemma below will be used in the sequel. 

\begin{lemma} \label{prime 1}
One has $\varphi(h_k(n) \cdot n)= h_k(n) \cdot \varphi(n)$ for each $n \geq 2$.
\end{lemma}

\begin{lemma} \label{prime 2}
Given $j \in \{1,\dots,r_1\}$, one has
$$\bigcup_{l=0}^{k-1} {\rm{Gal}}(L_k/F(e^{2i\pi l/k} \sqrt[k]{t_j})) = {\rm{Gal}}(L_k/F( e^{2i \pi /(h_k(n_j) \cdot n_j)})).$$
\end{lemma}

\begin{proof}
Let $j \in \{1,\dots, r_1\}$ and $l \in \{0,\dots,k-1\}$. From our choice of $\sqrt[k]{t_j}$ and since ${\rm{gcd}}(m_j,n_j)=1$, one has 
\begin{align*}
F(e^{2i\pi l/k} \sqrt[k]{t_j}) = F(e^{2 i \pi (l n_j+m_j) / (k \cdot n_j)}) 
&= F(e^{2 i \pi \cdot {\rm{gcd}} (l n_j + m_j, k \cdot n_j) / (k \cdot n_j)}) \\
&= F(e^{2 i \pi \cdot {\rm{gcd}} (l n_j + m_j,k) / (k \cdot n_j)}).
\end{align*}
Obviously, ${\rm{gcd}} (l n_j + m_j,k)$ divides $k = h_k(n_j) \cdot ({k}/{h_k(n_j)}).$ As $m_j$ and $n_j$ are coprime, the same is true for ${\rm{gcd}} (l n_j + m_j,k)$ and $h_k(n_j)$. Hence ${\rm{gcd}} (l n_j + m_j,k)$ divides $k/h_k(n_j)$. Then
$$
F(e^{2i \pi /(h_k(n_j) \cdot n_j)}) = F(e^{2 i \pi (k /h_k(n_j)) / (k \cdot n_j)})
\subseteq F(e^{2 i \pi \cdot {\rm{gcd}} (l n_j + m_j,k) / (k \cdot n_j)}).$$
Hence $F(e^{2i \pi /(h_k(n_j) \cdot n_j)}) \subseteq F(e^{2i\pi l/k} \sqrt[k]{t_j}).$ This provides $$\bigcup_{l=0}^{k-1} {\rm{Gal}}(L_k/F(e^{2i\pi l/k} \sqrt[k]{t_j})) \subseteq {\rm{Gal}}(L_k/F( e^{2i \pi /(h_k(n_j) \cdot n_j)})).$$

For the converse, it suffices to find $l_0 \in \{0,\dots,k-1\}$ such that $k/h_k(n_j)$ divides ${\rm{gcd}}(l_0 n_j + m_j, k)$. First, assume that $k$ has a prime factor not dividing $n_j$. By the Chinese Remainder Theorem, there exists $l_0 \in \{0, \dots, k-1\}$ such that 
$l_0 \equiv -m_j/n_j \, \, {\rm{mod}} \, \, p^{v_{p}(k)}$
for each prime factor $p$ of $k$ not dividing $n_j$. Hence $k/h_k(n_j)$ divides $l_0 n_j + m_j$, as needed. Now, assume that each prime factor of $k$ divides $n_j$. One then has $k/h_k(n_j) = 1$, thus ending the proof.
\end{proof}

\begin{lemma} \label{prime 3}
One has 
$$f_1(k) \leq \sum_{j=1}^{r_1}  \frac{[F:\Qq]}{d_1 \cdot h_k(n_{j})}.$$
\end{lemma}

\begin{proof}
By  Lemma \ref{prime 2} and the definitions of $f_1(k)$ and $d_1$, one has 
\begin{align*}
f_1(k) &=\frac{|\bigcup_{j=1}^{r_1} {\rm{Gal}}(L_k/F( e^{2i \pi /(h_k(n_j) \cdot n_j)}))|}{d_1 \cdot |\Gal(L_k/F)|}  \\
&\leq  \sum_{j=1}^{r_1} \frac{|{\rm{Gal}}(L_k/F(e^{2i \pi /(h_k(n_{j}) \cdot n_{j})}))|}{d_1 \cdot |{\rm{Gal}}(L_k/F)|} \\
&=\sum_{j=1}^{r_1} \frac{1}{d_1 \cdot [F(e^{2i \pi /(h_k(n_{j}) \cdot n_{j})}):F]}.
\end{align*}
For every positive integer $n$, one has $$[F(e^{2 i \pi /n}) : F] \geq \frac{[\Qq(e^{2 i \pi /n}):\Qq]}{[F:\Qq]} = \frac{\varphi(n)}{[F:\Qq]}.$$
We then get
$$f_1(k) \leq \sum_{j=1}^{r_1}  \frac{[F:\Qq]}{d_1 \cdot \varphi(h_k(n_{j}) \cdot n_{j})}.$$
It then remains to apply Lemma \ref{prime 1} to finish the proof.
\end{proof}

Let $(p_1, \dots, p_{r_1}) \in \mathcal{S}_1 \times \cdots \times \mathcal{S}_{r_1}$ and $\epsilon >0$. Given an integer $A \geq 1$, assume ${\rm{lcm}}  (p_1^A, \dots, p_{r_1}^A) \vert k$. One then has
$h_k(n_{j}) \geq p_{j}^{v_{p_{j}}(k)} \geq p_{j}^A \geq p_{\rm{min}}^A$
for each $j \in \{1,\dots,r_1\}$. Then apply Lemma \ref{prime 3} to get 
$$f_1(k) \leq \sum_{j=1}^{r_1} \frac{[F: \Qq]}{d _1 \cdot p_{\rm{min}}^A} = \frac{[F: \Qq] \cdot r_1}{d_1 \cdot p_{\rm{min}}^A}.$$
It then suffices to take $$A > \frac{\log ([F:\Qq]) + \log(r_1) - \log(d_1) - \log(\epsilon)}{\log(p_{\rm{min}})}$$
to get $f_1(k) < \epsilon$, thus ending the proof of part (2) of Theorem \ref{thm 4}. 

\subsection{Proof of part (3) of Theorem \ref{thm 4}}

Assume $r_2 + r_3 >0$.

\subsubsection{Refining the condition $f_2(k) <1$}

\begin{lemma} \label{prime 4}
One has $f_2(k) < 1$ if $$g_2(k):=\frac{|\bigcup_{j=r_1+1}^{r} {\rm{Gal}}(L_k/L_1(e^{2i\pi/k}, \sqrt[k]{t_j}))|}{|{\rm{Gal}}(L_k/L_1(e^{2i\pi/k}))|}<1.$$
\end{lemma}

\begin{proof}
Assume that there exists some $\sigma$ in 
${\rm{Gal}}(L_k/L_1(e^{2i \pi /k}))$ which is not in $\bigcup_{j=r_1 +1}^{r} {\rm{Gal}}(L_k/L_1(e^{2i\pi/k}, \sqrt[k]{t_j}))$. Then such an element $\sigma$ lies in 
$$\bigcup_{j=r_1 + 1}^{r} {\rm{Gal}}(L_k/F(t_j)) \, \setminus \, \bigcup_{j=r_1 + 1}^{r} \bigcup_{l=0}^{k-1} {\rm{Gal}}(L_k/F(e^{2i\pi l/k} \sqrt[k]{t_j})).$$
This provides $f_2(k) <1$, as needed for the lemma.
\end{proof}

Next, we need the following conditional bound.

\begin{lemma} \label{prime 5}
Assume that the polynomials $T^k-t_{r_1 + 1}, \dots, T^k - t_{r_ 1 +r_2}, T^k- t_{r_1 +r_2+1}, \dots, T^k-t_{r}$ all are irreducible over $L_1(e^{2i\pi/k})$. Then one has
$$g_2(k) \leq \frac{r_2+r_3}{k}.$$
\end{lemma}

\begin{proof}
By the definition of $g_2(k)$, one has 
\begin{align*}
g_2(k) &\leq \sum_{j=r_1 + 1}^{r} \frac{|{\rm{Gal}}(L_k/L_1(e^{2i\pi/k}, \sqrt[k]{t_{j}}))|}{|{\rm{Gal}}(L_k/L_1(e^{2i\pi/k}))|} \\
&= \sum_{j=r_1+1}^{r}  \frac{1}{[L_1(e^{2i\pi/k}, \sqrt[k]{t_{j}}):L_1(e^{2i\pi/k})]}.
\end{align*}
For each $j \in \{r_1+1, \dots,r\}$, one has 
$[L_1(e^{2i\pi/k}, \sqrt[k]{t_{j}}):L_1(e^{2i\pi/k})]= k$
as $T^k-t_{j}$ is irreducible over $L_1(e^{2i\pi/k})$. We then get
$$g_2(k) \leq \sum_{j=r_1+1}^{r} \frac{1}{k}= \frac{r-r_1}{k},$$
thus ending the proof.
\end{proof}

\subsubsection{On the irreducibility of the polynomials $T^k-t_{r_1+1}, \dots, T^k - t_{r_1 + r_2}, T^k- t_{r_1 + r_2+1}, \dots,$ $T^k-t_{r}$.} We start with the case where $t_j$ is a unit of $O_{L_1}$.

\begin{lemma} \label{prime 6}
Assume $r_2 >0$ and let $j \in \{r_1+1, \dots,r_1 + r_2\}$. The polynomial $T^k-t_j$ is irreducible over $L_1(e^{2 i \pi /k})$ if $k$ is a prime number $\geq k_0$ not dividing $a_j$.
\end{lemma}

\begin{proof}
First, assume that $T^k-t_j$ is reducible over $L_1$. By the Capelli lemma \cite[Chapter VI, \S9, Theorem 9.1]{Lan02} and as $k$ is a prime number, there exists $x \in O_{L_1}$ such that 
$
t_j=x^k.
$
As $t_j$ is a unit of $O_{L_1}$, the same is true for $x$. Set 
$$
x= \zeta' \cdot u_1^{w_1} \cdots u_{v}^{w_{v}},
$$ where $\zeta'$ is a root of unity and ${w_1}, \dots, w_{v}$ are integers. We then get $$\zeta_j \cdot u_1^{a_{j,1}} \cdots u_{v}^{a_{j,v}} = t_j=x^{k}=\zeta'^{k} \cdot u_1^{k \cdot w_1} \cdots u_{v}^{k \cdot w_{v}}.$$ In particular, we get $a_{j,l}= k \cdot w_l$ for each $l \in \{1,\dots,v\}$. Then $k$ divides $a_j$, which cannot happen. Hence $T^k-t_j$ is irreducible over $L_1$.

Now, we show that $T^k-t_j$ is irreducible over $L_1(e^{2 i \pi / k})$. By the definition of $k_0$ and as $k$ is a prime number $\geq k_0$, the fields $L_1$ and $\Qq(e^{2 i \pi / k})$ are linearly disjoint over $\Qq$, {\it{i.e.}}, one has 
$$[L_1(e^{2 i \pi / k}):L_1]=[\Qq(e^{2 i \pi / k}):\Qq]=k-1$$ (since $k$ is a prime number). By the above, one has $[L_1(\sqrt[k]{t_j}):L_1]=k.$ Since $k$ and $k-1$ are coprime, the fields $L_1(\sqrt[k]{t_j})$ and $L_1(e^{2 i \pi / k})$ are linearly disjoint over $L_1$. Hence we get $$[L_1(e^{2 i \pi / k},\sqrt[k]{t_j}): L_1(e^{2 i \pi / k})]=[L_1(\sqrt[k]{t_j}):L_1]=k,$$ as needed for the lemma.
\end{proof}

Now, we consider the case where $t_j$ is not a unit of $O_{L_1}$.

\begin{lemma} \label{prime 7}
Suppose $r_3 >0$ and let $j \in \{r_1 + r_2+1, \dots, r\}$. Then $T^k-t_j$ is irreducible over $L_1(e^{2 i \pi /k})$ if $k$ is a prime not dividing $v_{\mathcal{P}_j}(t_j)$.
\end{lemma}

\begin{proof}
Assume that $T^k-t_j$ is reducible over $L_1(e^{2 i \pi /k})$. By the Capelli lemma and as $k$ is a prime, there exists $x \in O_{L_1(e^{2 i \pi /k})}$ such that $t_j=x^{k}.$ Let $\mathcal{Q}_j$ be a non-zero prime ideal of $O_{L_1(e^{2 i \pi /k})}$ lying over $\mathcal{P}_j$. One has
$$v_{\mathcal{Q}_j}(t_j) = k \cdot  v_{\mathcal{Q}_{j}}(x).$$ This gives $e_j \cdot v_{\mathcal{P}_j}(t_j) = k \cdot  v_{\mathcal{Q}_{j}}(x)$, where $e_j$ is the ramification index of $\mathcal{P}_j$ in $L_1(e^{2 i \pi /k})/L_1$. As $v_{{\mathcal{P}_{j}}}(t_j)$ is a positive integer, this is also true for $v_{\mathcal{Q}_j}(x)$. Hence the prime $k$ divides $e_j$ or $v_{{\mathcal{P}_{j}}}(t_j)$. As $e_j \leq k-1$, the prime $k$ has to divide $ v_{{\mathcal{P}_{j}}}(t_j)$, which cannot happen.
\end{proof}

\subsubsection{Conclusion}
For simplicity, assume $r_2 >0$ and $r_3 >0$ (the other two cases are similar). Suppose $k$ is a prime number $\geq c$. As $k$ satisfies $k \geq k_0$ and $k {\not \vert} \, {\rm{lcm}} (a_0,v_0)$, one may apply Lemmas \ref{prime 6} and \ref{prime 7} to get that the polynomials $T^k-t_{r_1+1}, \dots, T^k-t_{r_1+r_2}, T^k-t_{r_1+r_2+1}, \dots, T^k-t_{r}$ all are irreducible over $L_1(e^{2 i \pi / k})$. Then, by Lemma \ref{prime 5} and since $k > r_2+r_3$, we get $g_2(k) <1$. It then remains to apply Lemma \ref{prime 4} to finish the proof of part (3) of Theorem \ref{thm 4}.

\begin{remark}
More generally, the proof shows that the condition $g_2(k) <1$ (and then the condition $f_2(k) <1$ too) holds if $k$ is a prime such that

\noindent
- $k \geq {\rm{max}} \, (r_2+r_3+1,k_0)$ and $k {\not \vert} \, {\rm{lcm}}  (a_0,v_0)$ if $r_2 >0$ and $r_3 >0$,

\noindent
- $k \geq r_3+1$ and $k {\not \vert} \, v_0$ if $r_2 =0$,

\noindent
- $k \geq {\rm{max}} \, (r_2+1,k_0)$ and $k {\not \vert} \, a_0$ if $r_3=0$.
\end{remark}

\section{On the converse of Theorem \ref{thm 2}}

In Propositions \ref{converse} and \ref{converse 2} below, we show that the conclusion of Theorem \ref{thm 2} does not hold in general if $P(T)$ has a root that is a root of unity. This suggests that our strategy to handle the roots of $P(T)$ that are not roots of unity, which leads to a better conclusion in Theorem \ref{thm 2}, cannot be extended to the case of roots of unity.

\begin{proposition} \label{converse}
Assume that the following condition holds:

\vspace{1mm}

\noindent
${\rm{(H)}} \bigcup_{j=r_1+1}^{r} \Gal(L_1/F(t_j)) \subset \bigcup_{j=1}^{r_1} \Gal(L_1/F(t_j)).$

\vspace{1mm}

\noindent
Then $r_1 >0$ and condition {\rm{($*$/$k$)}} of Proposition 1 fails for each positive integer $k$ which is coprime to ${\rm{lcm}}(n_1, \dots,n_{r_1})$. In particular, the conclusion of Theorem \ref{thm 2} does not hold.
\end{proposition}

\noindent
For example, condition (H) holds in each of the following situations:

\vspace{0.5mm}

\noindent
{\rm{(1)}} each root of $P(T)$ is a root of unity,

\vspace{0.5mm}

\noindent
{\rm{(2)}} $P(T)$ has a root that is a root of unity and that is in $F$.

\vspace{0.5mm}

\noindent
Indeed, in case (1), the left-hand side in condition (H) is empty while, in case (2), the right-hand side is equal to the whole group ${\rm{Gal}}(L_1/F)$.

\begin{proof}
Assume that condition ${\rm{(H)}}$ holds. Then one has $r_1>0$ as the polynomial $P(T)$ is not constant. Let $k \geq 1$ be an integer such that $k$ and ${\rm{lcm}}(n_1, \dots,n_{r_1})$ are coprime. By Lemma \ref{prime 2}, one has
$$\bigcup_{j=1}^{r_1} \bigcup_{l=0}^{k-1} {\rm{Gal}}(L_k/F(e^{2i\pi l/k} \sqrt[k]{t_j})) = \bigcup_{j=1}^{r_1} {\rm{Gal}}(L_k/F( e^{2i \pi /(h_k(n_j) \cdot n_j)})).$$
By the assumption on $k$, one has $$h_k(n_1)= \cdots= h_k(n_{r_1}) =1.$$ This gives
$$\bigcup_{j=1}^{r_1} \bigcup_{l=0}^{k-1} {\rm{Gal}}(L_k/F(e^{2i\pi l/k} \sqrt[k]{t_j})) = \bigcup_{j=1}^{r_1} {\rm{Gal}}(L_k/F( e^{2i \pi / n_j})) =  \bigcup_{j=1}^{r_1} {\rm{Gal}}(L_k/F( t_j)).$$
Moreover, as condition ${\rm{(H)}}$ holds, one has
$$\bigcup_{j=r_1 + 1}^{r} \bigcup_{l=0}^{k-1} {\rm{Gal}}(L_k/F(e^{2i\pi l/k} \sqrt[k]{t_j})) \subseteq \bigcup_{j=r_1 + 1}^{r} {\rm{Gal}}(L_k/F(t_j)) \subseteq \bigcup_{j=1}^{r_1} {\rm{Gal}}(L_k/F(t_j)).$$
We then get 
$$\bigcup_{j=1}^{r} \bigcup_{l=0}^{k-1} {\rm{Gal}}(L_k/F(e^{2i\pi l/k} \sqrt[k]{t_j})) =  \bigcup_{j=1}^{r_1} {\rm{Gal}}(L_k/F(t_j)) = \bigcup_{j=1}^{r} {\rm{Gal}}(L_k/F(t_j)).$$
(as condition (H) holds). Hence $f(k)=1$. It then remains to apply part (1) of Theorem \ref{thm 4} to finish the proof.
\end{proof}

In Proposition \ref{converse 2} below, we give a more precise conclusion in the previous case (2). 

\begin{proposition} \label{converse 2}
Assume that $P(T)$ has a root that is a root of unity and that is in $F$. Denote the number of such roots by $r'_1$ and assume that $t_1, \dots, t_{r'_1}$ are in $F$. Then condition {\rm{($*$/$k$)}} of Proposition 1 fails for each integer $k \geq 1$ that is coprime to $n_j$ for some $j \in \{1,\dots,r'_1\}$.
\end{proposition}

\noindent
In particular, if $P(T)$ has at least two roots that are roots of unity, that are in $F$ and that have coprime orders, then condition {\rm{($*$/$k$)}} of Proposition 1 fails for all prime numbers $k$.

\begin{proof}
Let $k$ be a positive integer. Assume that $k$ is coprime to $n_{j_0}$ for some $j_0 \in \{1,\dots,r'_1\}$. By Lemma \ref{prime 2}, one has
$$\bigcup_{l=0}^{k-1} {\rm{Gal}}(L_k/F(e^{2i\pi l/k} \sqrt[k]{t_{j_0}})) = {\rm{Gal}}(L_k/F( e^{2i \pi /(h_k(n_{j_0}) \cdot n_{j_0})})).$$
As $k$ and $n_{j_0}$ are coprime, one has $h_k(n_{j_0})=1$. This gives
$${\rm{Gal}}(L_k/F( e^{2i \pi /(h_k(n_{j_0}) \cdot n_{j_0})})) = {\rm{Gal}}(L_k/F( e^{2i \pi / n_{j_0}})) ={\rm{Gal}}(L_k/F)$$
(as $t_{j_0}$ is in $F$). Hence 
$$
\bigcup_{j=1}^{r} \bigcup_{l=0}^{k-1} {\rm{Gal}}(L_k/F(e^{2i\pi l/k} \sqrt[k]{t_j})) = {\rm{Gal}}(L_k/F).
$$
Moroever, as $t_{j_0}$ is in $F$, one has $\bigcup_{j=1}^{r} {\rm{Gal}}(L_k/F(t_j)) = {\rm{Gal}}(L_k/F).$
This provides $f(k)=1$, thus ending the proof.
\end{proof}

\bibliography{Biblio2}

\begin{thebibliography}{Leg16b}

\bibitem[Bec91]{Bec91}
Sybilla Beckmann.
\newblock On extensions of number fields obtained by specializing branched
  coverings.
\newblock {\em J. Reine Angew. Math.}, 419:27--53, 1991.

\bibitem[D{\`e}b16]{Deb16}
Pierre D{\`e}bes.
\newblock Groups with no parametric {G}alois realizations.
\newblock 2016.
\newblock To appear in Annales scientifiques de l'{\'E}cole normale
  sup\'erieure. arXiv 1605.09363.

\bibitem[GB71]{GB71}
Irving Gerst and John Brillhart.
\newblock On the prime divisors of polynomials.
\newblock {\em Amer. Math. Monthly}, 78(3):250--266, 1971.

\bibitem[Gra07]{Gra07}
Andrew Granville.
\newblock Rational and integral points on quadratic twists of a given
  hyperelliptic curve.
\newblock {\em Int. Math. Res. Not. IMRN}, no. 8, 2007.
\newblock Art. {I}{D} 027, 24 pp.

\bibitem[Lan02]{Lan02}
Serge Lang.
\newblock {\em Algebra}, volume 211 of {\em Graduate {T}exts in {M}athematics}.
\newblock Springer-Verlag, New York, revised third edition, 2002.

\bibitem[Leg16a]{Leg16a}
Fran\c{c}ois Legrand.
\newblock On parametric extensions over number fields.
\newblock 2016.
\newblock To appear in Annali della Scuola Normale Superiore di Pisa - Classe
  di Scienze. arXiv 1602.06706.

\bibitem[Leg16b]{Leg16c}
Fran\c{c}ois Legrand.
\newblock Specialization results and ramification conditions.
\newblock {\em Israel J. Math.}, 214(2):621--650, 2016.

\bibitem[Leg16c]{Leg16b}
Fran\c{c}ois Legrand.
\newblock Twists of superelliptic curves without rational points.
\newblock {\em International Mathematics Research Notices}, 2016.
\newblock doi: 10.1093/imrn/rnw270.

\bibitem[Nag69]{Nag69}
Trygve Nagell.
\newblock Sur les diviseurs premiers des polyn\^omes.
\newblock {\em Acta Arith.}, 15:235--244, 1969.

\bibitem[Sad14]{Sad14}
Mohammad Sadek.
\newblock On quadratic twists of hyperelliptic curves.
\newblock {\em Rocky Mountain J. Math.}, 44(3):1015--1026, 2014.

\bibitem[Sch69]{Sch69}
Andrzej Schinzel.
\newblock Remarque sur le travail pr\'ec\'edent de {T}. {N}agell.
\newblock {\em Acta Arith.}, 15:245--246, 1969.

\end{thebibliography}
\bibliographystyle{alpha}

\end{document}